\documentclass[12pt]{amsart}
\usepackage{amsmath}
\usepackage{amssymb}
\usepackage{epsfig}
\usepackage{graphicx}
\usepackage{xcolor}
\numberwithin{equation}{section}

\input xy
\xyoption{all}

\calclayout
\allowdisplaybreaks[3]

\theoremstyle{plain}
\newtheorem{prop}{Proposition}

\newtheorem{coro}[prop]{Corollary}

\theoremstyle{definition}

\newtheorem{ques}[prop]{Question}

\newtheorem{rema}[prop]{Remark}

\def\ra{\rightarrow}

\def\cM{{\mathcal M}}

\def\cO{{\mathcal O}}
\def\cP{{\mathcal P}}
\def\cQ{{\mathcal Q}}
\def\cR{{\mathcal R}}

\def\cU{{\mathcal U}}
\def\cV{{\mathcal V}}

\def\fS{{\mathfrak S}}

\def\bA{{\mathbb A}}
\def\bG{{\mathbb G}}
\def\bP{{\mathbb P}}
\def\bQ{{\mathbb Q}}

\def\bZ{{\mathbb Z}}

\def\rH{{\mathrm H}}

\def\fS{{\mathfrak S}}
\def\ft{{\mathfrak t}}

\def\GL{\mathrm{GL}}

\def\Bl{\mathrm{Bl}}
\def\Pic{\mathrm{Pic}}
\def\End{\mathrm{End}}

\def\Gal{\mathrm{Gal}}
\def\Aut{\mathrm{Aut}}
\def\Aff{\mathrm{Aff}}
\def\Gr{\mathrm{Gr}}

\def\NS{\mathrm{NS}}

\def\Ext{\mathrm{Ext}}
\def\Hom{\mathrm{Hom}}

\def\lim{\mathrm{lim}}

\def\Cox{\mathrm{Cox}}
\def\Spec{\mathrm{Spec}}

\def\rank{\mathrm{rk}}
\def\PGL{\mathrm{PGL}}

\def\fA{{\mathfrak A}}

\makeatother
\makeatletter

\author{Brendan Hassett}
\address{Department of Mathematics\\
Brown University \\
Box 1917 
151 Thayer Street
Providence, RI 02912 \\
USA}
\email{brendan\underline{ }hassett@brown.edu}

\author{Yuri Tschinkel}
\address{Courant Institute\\
                New York University \\
                New York, NY 10012 \\
                USA }
\email{tschinkel@cims.nyu.edu}

\address{Simons Foundation\\
160 Fifth Avenue\\
New York, NY 10010\\
USA}

\title{Torsors and stable equivariant birational geometry}

\begin{document}
\date{\today}

\begin{abstract}
We develop the formalism of universal torsors in equivariant birational geometry and apply it to produce new examples of nonbirational but stably birational actions of finite groups.  

\end{abstract}

\maketitle

\section{Introduction}

Let $k$ be an algebraically closed field of characteristic zero.
Consider a finite group $G$, acting regularly on a smooth projective variety $X$ over $k$, generically freely from the right.  
Given two such varieties $X$ and $X'$ with $G$-actions, we say that $X$ and $X'$ are 
{\em $G$-birational}, and write
$$
X\sim_G X',
$$
if there is a $G$-equivariant
birational map
$$
X\stackrel{\sim}{\dashrightarrow} X'.
$$
We say that $X$ and $X'$ are 
{\em stably $G$-birational} if there is a $G$-equivariant
birational map
$$
X \times \bP^n \stackrel{\sim}{\dashrightarrow} X' \times \bP^{n'},
$$
where the action of $G$ on the projective spaces is trivial.  The {\em No-Name Lemma} implies that this is equivalent to the existence of 
$G$-equivariant vector bundles $E\rightarrow X$ and $E'\rightarrow X'$ that are $G$-birational to each other. In particular, faithful linear
actions on $\bA^n$ are always stably $G$-birational but not always $G$-birational \cite{RYinvariant}, \cite{KT-vector}.   We say that the $G$-action
on an $n$-dimensional variety $X$ is (stably) {\em linearizable} if there exists an $(n+1)$-dimensional faithful representation $V$ of $G$ such that $X$ is (stably) $G$-birational to $\bP(V)$.

There are  a number of tools to distinguish $G$-birational actions, including
\begin{itemize}
\item{existence of fixed points upon restriction to abelian subgroups of $G$ \cite{RYessential};}
\item{determinant of the action of abelian subgroups in the tangent bundle at fixed points \cite{RYinvariant};}
\item{Amitsur group and $G$-linearizability of line bundles \cite[Section 6]{blancfinite};} 
\item{group cohomology for induced actions on invariants such as the N\'eron-Severi group \cite{BogPro};}
\item{equivariant birational rigidity, see, e.g., \cite{CS};}
 \item{equivariant enhancements of intermediate Jacobians and cycle invariants \cite{HT-intersect};}
\item{equivariant Burnside groups \cite{BnG}, \cite{KT-vector}.}
\end{itemize}
Of these, only the fixed point condition for abelian subgroups, the Amitsur group, and group cohomology -- specifically $\rH^1(G, \Pic(X))$ or higher unramified cohomology --  yield {\em stable} $G$-birational invariants.

Nevertheless, nontrivial stable birational equivalences are hard to come by. In this paper, we adopt the formalism of universal torsors -- developed
by Colliot-Th\'el\`ene, Sansuc, Skorobogatov, and others, in the context of arithmetic questions like Hasse principle and weak approximation
-- to the framework of equivariant birational geometry. As an application, we exhibit new examples of nonbirational but stably birational actions. Specifically, we
\begin{itemize}

\item show that the linear $\fS_4$-action on $\bP^2$ and an $\fS_4$-action on a del Pezzo surface 
of degree 6 are not birational but stably birational (Proposition~\ref{prop:dp6-s4}),
\item settle the stable linearizability problem for quadric surfaces (Proposition~\ref{prop:quads}),
\item show that the linear $\fA_5$-action on $\bP^2$ 
and the natural $\fA_5$-action on a del Pezzo surface of degree 5 are not birational but stably birational (Proposition~\ref{prop:a5}),
\item show that $\fA_5$-actions on the Segre cubic threefold, arising from two nonconjugate embeddings of $\fA_5\hookrightarrow \fS_6$, are not birational but stably birational (Proposition~\ref{prop:segre}).
\end{itemize}

\

Here is the roadmap of the paper: In Sections~\ref{sect:tori} and \ref{sect:equi-for} we extend the formalism of universal torsors and Cox rings to the context of equivariant geometry over $k$. In Section~\ref{sect:torvar}, we study the (stable) linearization problem for toric varieties. A key example, del Pezzo surfaces of degree six, is discussed in
Section~\ref{sect:dp6}; the related case of Weyl group actions for $\mathsf{G}_2$ is presented in 
Section~\ref{sect:wg2}. In Section~\ref{sect:quad} we turn to quadric surfaces.
In Section~\ref{sect:quotor} we discuss linearization of actions of Weyl groups on Grassmannians and their quotients by tori.

\

\noindent
{\bf Acknowledgments:} 
The first author was partially supported by Simons Foundation Award 546235 and NSF grant 1701659,
the second author by NSF grant 
2000099. 

\section{Algebraic tori and torsors over nonclosed fields}
\label{sect:tori}

Let $k$ be a field of characteristic zero and $X$ an $d$-dimensional geometrically rational variety over $k$. Recall that $X$ is called (stably) $k$-rational if $X$ is (stably) birational to $\bP^d$ over $k$. 

An important class of varieties which was studied from the perspective of (stable) $k$-rationality is that of algebraic tori. A classification of (stably) $k$-rational tori in dimensions $d\le 5$ can be found in \cite{voskresenskii2dim}, \cite{kun}, \cite{hoshi}.    

In this section, we review the main features of the theory of tori and torsors under tori over nonclosed fields. Our main references are \cite{Sansuc-CT} and \cite{CTSansucDuke}. 

\subsection{Characters and Galois actions}
\label{sect:char}

Recall that an algebraic torus $T$ over $k$ is an algebraic group over $k$ such that 
$$
\bar{T}:=T_{\bar{k}} = \mathbb G_m^d, 
$$
over an algebraic closure $\bar{k}$ of $k$. 
Let $M$ 
be its character lattice and $N$ the lattice of cocharacters,
which carry actions of the absolute Galois group $\Gal(k)$ of $k$.

The descent data for a torus $T$ over an arbitrary field $k$ of characteristic zero is encoded by the
continuous representation
$$
\Gal(k) \rightarrow \GL(M).
$$

\subsection{Quasi-trivial tori}
\label{sect:quasi}

There is a tight connection between (stable) $k$-rationality of  $T$ and properties of the Galois module $M$.

Recall  that $M$ is called a {\em permutation} module if $M$ has a $\bZ$-basis permuted by $\Gal(k)$, i.e., $M$ is a direct sum of modules of the form $\bZ[\Gal(k)/H]$, where $H$ is a closed finite-index subgroup. By definition, a torus $T$ is {\em quasi-trivial} if $M$ is a permutation module. 
Quasi-trivial tori are rational over $k$ by Hilbert's Theorem 90 for general linear groups.

Every torus may be expressed as a subtorus or quotient of a quasi-trivial torus, by expressing  the
character or cocharacter lattices as quotients of permutation modules.  

\subsection{Rationality criteria} 
\label{sect:ratcrit}

A fundamental theorem \cite{vosk} is that a torus $T$ is stably rational if and only if $M$ 
is {\em stably permutation}, i.e., there exist permutation modules $P$ and $Q$ such that
$$
M\oplus P \simeq Q.
$$
This condition implies the vanishing of
$$
\rH^1(H, M)
$$
for all closed finite-index subgroups $H\subseteq \Gal(k)$ (i.e., $M$ is {\em coflabby}).

\subsection{Torsor formalism}
\label{sect:torsors}

Let $X$ be a smooth projective geometrically rational variety over $k$. 
Since $\bar{X}$ is rational, $\Pic(\bar{X})\ra \NS(\bar{X})$ is an isomorphism.
Let 
$$
T_{\NS(\bar{X})}
$$ 
denote the N\'eron-Severi torus of $X$, i.e., a torus whose
character group 
is isomorphic, as a Galois module, to $\NS(\bar{X})$. 
Let 
$$
\cP \rightarrow X
$$ 
be a {\em universal torsor} for $T_{\NS(\bar{X})}$ over $k$; below we will discuss when it exists over the ground field.  
Recall that 
$
\cP \rightarrow X
$ 
is a morphism defined over $k$, admitting a free action
$$\cP \times T_{\NS(\bar{X})}\rightarrow \cP$$
over $X$ with the following geometric property: 
Choose a basis 
$$
\lambda_1,\ldots,\lambda_r \in \NS(\bar{X})=\Hom(T_{\NS(\bar{X})},\bG_m),
$$
so that the associated rank-one bundles
$L_1,\ldots,L_r \ra X$ satisfy $$
\lambda_i=[L_i], \quad i=1,\ldots,r.
$$
This determines $\cP$ uniquely over an algebraic closure $\bar{k}/k$; however for each $\gamma \in \rH^1(\Gal(k),T_{\NS(\bar{X})})$, we can twist the 
torus action to obtain another such torsor ${}^{\gamma}\cP$.  

Given a homomorphism of free Galois modules 
$$
\alpha: M \rightarrow \NS(\bar{X})
$$ 
there is a homomorphism of tori 
$T_{\NS(\bar{X})} \rightarrow T_M$ and an induced torsor $\cP_{\alpha} \rightarrow X$ for $T_M$.  

A sufficient condition for the existence of a universal torsor over $k$ is the existence of a 
$k$-rational point $x\in X(k)$: one can define $\cP \rightarrow X$ over $k$ via 
evaluation at $x$. More generally,
suppose that $D_1,\ldots,D_r$ is a collection of effective divisors on $\bar{X}$ that is Galois-invariant and
generates $\NS(\bar{X})$.  Let $U$ denote their complement in $X$; we have an exact sequence
$$0 \rightarrow R=\bar{k}[U]^\times/\bar{k}^\times \rightarrow \oplus_{j=1}^r \bZ D_j \rightarrow \NS(\bar{X}) \rightarrow 0.$$
The following conditions are equivalent \cite[Prop.~2.2.8]{CTSansucDuke}:
\begin{itemize}
\item{the short exact sequence
\begin{equation}
1 \rightarrow \bar{k}^\times \rightarrow \bar{k}[U]^\times \rightarrow \bar{k}[U]^\times/\bar{k}^\times \rightarrow 1 \label{U-SES} \end{equation}
splits;}
\item{the descent obstruction for $\bar{\cP}$ in $\rH^2(\Gal(k),T_{\NS(\bar{X})})$ vanishes.}
\end{itemize}
Indeed, each rational point $x\in U(k)$ gives a splitting of (\ref{U-SES}).

When can the universal torsor -- or more general torsor constructions -- be used to obtain stable rationality results for $X$ over $k$? 
\begin{prop} \label{prop:stabratcrit}
A smooth projective geometrically rational variety 
$X$ over $k$ is stably rational over $k$ under the following conditions: 
\begin{itemize}
\item{its universal torsor $\cP \ra X$ is rational over $k$;}
\item{its N\'eron-Severi torus $T_{\NS(\bar{X})}$ is stably rational;}
\item{the morphism $\cP \ra X$ admits a rational section, i.e., the torsor splits.}
\end{itemize}
\end{prop}
The last two conditions hold \cite[Prop.~3]{BCTSSD} if $\NS(\bar{X})$ is stably permutation.  
Note that there are examples where the relevant cohomology vanishes ($\NS(\bar{X})$ is flabby and coflabby)
but $\NS(\bar{X})$ fails to be a stable permutation module; these can be found in \cite[Remarque R4]{Sansuc-CT} (see also  \cite[Section 1]{hoshi}).

\section{Equivariant formalism}
\label{sect:equi-for}

We turn to the equivariant context, working over an 
algebraically closed field $k$ of characteristic zero. Our goal is to formulate a $G$-equivariant version of the torsor formalism in \cite{CTSansucDuke}, which will be our main tool in the study of the (stable) linearization problem.

\subsection{$G$-tori}
\label{sect:g-tori}

Let $T=\mathbb G_m^d$ be an algebraic torus over $k$.  Recall that we have a split exact sequence 
\begin{equation}
\label{eqn:aut-tor}
1\to T(k)\to \Aff(T)\to \Aut(T)\to 1,
\end{equation}
where $\Aut(T)$ is the automorphisms of $T$ as an algebraic group and $\Aff(T)$ is the associated affine group.
Note that $\Aut(T)$ acts faithfully on the character lattice of $T$.

Let $G\subset \Aut(T)$ be a finite group, so that $T$ is a group in the category of $G$-varieties. 
We refer to such tori as {\em $G$-tori}. Given $G\subset \Aff(T)$, the elements in $G\cap T(k)$ will be called {\em translations}. This gives rise to a torsor
$$P\times T \rightarrow P,$$
where $T$ is the $G$-torus associated with the composition $G\rightarrow \Aff(T) \rightarrow \Aut(T)$.

The (stable) linearization problem for $G$-tori concerns (stable) birationality of the $G$-action on $T$ and a linear $G$-action on $\bP^d$. There are two extreme cases:
\begin{itemize}
\item $G\subset T(k)$, i.e., $G$ is abelian and the $G$-action is a translation action,
\item $G\cap T(k)=1$.
\end{itemize}

\subsection{Linearizing translation actions}

An action of $G\subset T(k)$ extends to a linear action; indeed it extends to a
linear action on the natural compactification $T\hookrightarrow \bP^d$. Note that these do not have to be equivariantly birational to each other, for different embeddings $G\hookrightarrow T(k)$; the determinant condition of  \cite{RYinvariant} characterizes such actions up to equivariant birationality. By the No-Name Lemma, translation actions are stably equivariantly birational. For {\em nonabelian} $G$ containing an abelian subgroup of rank $d$, 
similar examples of nonbirational but stably birational $G$-actions on tori can be extracted from 
\cite[Prop. 7.2]{RYinvariant}.

\subsection{Linearizing translation-free actions}
\label{sect:free}

The (stable) linearization problem for actions without translations is essentially equivalent to the well-studied (stable) rationality problem of tori over nonclosed fields. It is
controlled by the $G$-action on the cocharacters. 
We record:

\begin{prop}
\label{prop:toric:mor}
Let $T$ be a $G$-torus (i.e., $G\cap T(k)=1$) with cocharacter module $N$. Assume that
$N$ is a stably permutation $G$-module. Then the $G$-action on $T$ is stably linearizable. 
\end{prop}

\begin{proof}
Suppose first that $N$ is a permutation module. We can realize our torus
$$T \subset \bA^d, \quad d=\dim(T),$$
as an open subset of affine space twisted by a permutation of the basis vectors.
Any linear twist of affine space is isomorphic to affine space by Hilbert's Theorem 90,
hence the $G$-action on $T$ is linearizable as well.

If $N$ is stably permutation then there exist permutation modules $P$ and $Q$ such that
$$N \oplus P \simeq Q.$$
The argument above yields
$$
T \times \bA^{\dim(P)} \stackrel{\sim}{\dashrightarrow} \bA^{\dim(Q)}
$$
which, combined with the No-Name Lemma, gives that the action is stably linear.  
\end{proof}

\begin{ques}
Can we effectively compute whether a $G$-module is stably permutation? 
\end{ques}


\subsection{$G$-equivariant torsors}
\label{sect:g-torsors}

We now turn to general smooth projective varieties $X$ with a generically free regular action of a finite group $G$. We assume that 
$$
\NS(X)=\Pic(X)
$$ 
is a free abelian group; it inherits the $G$-action.  Let 
$$
T_{\NS(X)}:=\Hom(\NS(X),\bG_m)
$$
denote the N\'eron-Severi torus, it is a $G$-torus. 

Let $T$ be a $G$-torus with character module $\hat{T}$. A $G$-equivariant $T$-torsor over $X$ consists of a $G$-equivariant
scheme $\cP \ra X$ and a $G$-equivariant action 
$$
\cP\times T \rightarrow \cP
$$ 
over $X$ that is a torsor on the underlying groups and varieties. 
Let 
$$
\rH^1_G(X,T)
$$ 
denote the group of isomorphism classes of $G$-equivariant 
$S$-torsors 
over $X$.  We have an exact sequence
\begin{equation} \label{LES}
0 \rightarrow \rH^1(G,T) \rightarrow \rH^1_G(X,T) \rightarrow \Hom_G(\hat{T},\Pic(X)) 
\stackrel{\partial}{\rightarrow} \rH^2(G,T).
\end{equation}
The middle arrow may be understood as recording the line bundles arising from characters of $T$.

\subsection{Amitsur group}
\label{sect:amitsur}
Restricting to $G$-invariant divisors 
$$
\Pic(X)^G \subset \Pic(X),
$$
we obtain
$$0 \rightarrow \Hom(G,\bG_m) \rightarrow \Pic_G(X) \rightarrow \Pic(X)^G \rightarrow \rH^2(G,\bG_m)$$
where $\Pic_G(X)$ is the group of $G$-linearized line bundles on $X$.
The class 
$$
\alpha=\partial([h]),
$$
where $h$ is $G$-invariant,
is called the {\em Schur multiplier}. It vanishes if and only if the 
$G$-action lifts to $\Gamma(X,\cO_X(mh))$ for
each $m>0$. The subgroup 
$$
\mathrm{Am}(X,G)\subseteq \rH^2(G,\bG_m)
$$ 
generated by all such classes is called the {\em Amitsur group} \cite[\S 6]{blancfinite};
it is a stable $G$-birational invariant \cite[Thm.~2.14]{sarikyan}.
Note that when $\mathrm{Am}(X,G)=0$ there may be subgroups $H \subsetneq G$ with
$\mathrm{Am}(X,H) \neq 0$.

\subsection{Lifting the $G$-action}

Suppose that 
$$
\cP\rightarrow X
$$ 
is a {\em universal torsor}, i.e., a torsor for $T=T_{\NS(X)}$ whose class
in $\Hom(\hat{T},\Pic(X))$ is the identity. When does the $G$-action on $X$ lift to $\cP$? This problem is analogous to the problem of descending 
the universal torsor to the ground field, in the arithmetic context of Section~\ref{sect:torsors}. 

Here are two sufficient conditions:
\begin{itemize}
\item{$X$ admits a $G$-fixed point;}
\item{the cocycle 
$$
\alpha=\partial (\operatorname{Id}) \in \rH^2(G,T_{\NS(X)})
$$ 
vanishes
(whence all Schur multipliers are trivial).}
\end{itemize}
The latter is necessary by the long exact sequence (\ref{LES}).
The following proposition gives a criterion for the vanishing of this cocycle:
\begin{prop} \label{prop:openinduce}
Let $X$ be a smooth projective $G$-variety. Assume that  $\Pic(X)$ is a free abelian group.
Fix a $G$-invariant open subset $\emptyset \neq U \subset X$ with $\Pic(U)=0$.
The class $\alpha \in \rH^2(G,T_{\NS(X)})$ vanishes if and only if the exact sequence
\begin{equation} \label{eqn:splitU}
1 \rightarrow k^\times \rightarrow k[U]^\times \rightarrow k[U]^\times/k^\times \rightarrow 1
\end{equation}
has a $G$-equivariant splitting. 
\end{prop}
The proof is completely analogous to the proof of \cite[2.2.8(v)]{CTSansucDuke} with group
cohomology replacing Galois cohomology.

\subsection{Constructing the torsor}

This approach can yield a construction for the universal torsor. Let $D_1,\ldots,D_r$ be a 
$G$-invariant collection of effective divisors generating
$\Pic(X)$.  The complement 
$$
U=X\setminus (D_1\cup\ldots \cup D_r)
$$
has trivial Picard group. Consider the exact
sequence
$$ 0 \rightarrow \hat{R} \rightarrow \oplus_{i=1}^r \bZ D_i \rightarrow \Pic(X) \rightarrow 0,
$$
where $\hat{R}$ is the module of relations among the $D_i$,
and its dual
\begin{equation} \label{torusSES}
0 \rightarrow T_{\NS(X)} \rightarrow M \rightarrow R \rightarrow 0.
\end{equation}
There is a canonical $G$-homomorphism
$$
\hat{R} \rightarrow k[U]^\times/k^\times
$$
obtained by regarding the relations as rational functions that are invertible on $U$. 
The existence of a splitting for (\ref{eqn:splitU}) yields a lift
$$
\hat{R} \rightarrow k[U]^\times,
$$
whence a morphism
$$U \rightarrow R.$$
The sequence (\ref{torusSES}) induces a $T_{\NS(X)}$-torsor over $U$, which 
extends to all of $X$ as in ~\cite[Thm.~2.3.1]{CTSansucDuke}.

\subsection{Properties of torsors}

We also have the equivariant version of \cite[Prop.~3]{BCTSSD}, an application of Hilbert Theorem 90:
\begin{prop} \label{prop:split}
Suppose $\NS(X)$ is stably permutation as a $G$-module. If $\cP\ra X$ is a universal torsor then
there exists a $G$-equivariant rational section $X\dashrightarrow \cP$, whence 
$$
\cP \sim_G T_{\NS(X)}\times X.
$$
\end{prop}

\begin{coro}
\label{g-bir}
The existence of a $G$-equivariant universal torsor is a $G$-birational property.
\end{coro}

\begin{proof}
Indeed, if $X$ and $Y$ are $G$-equivariantly birational then we can exhibit an affine open
subset common to both varieties for which Proposition~\ref{prop:openinduce} applies.
\end{proof}

In parallel with \cite[Prop. 2.9.2]{CTSansucDuke}, we have: 

\begin{prop}
The existence of a $G$-equivariant universal torsor is a stable $G$-birational property.
\end{prop}

\begin{proof}
Let $W$ be a smooth projective $G$-variety, equivariantly birational to a linear generically-free
action on projective space. Then $\Pic(W)$ is stably a permutation module and each invariant line bundle on $W$
admits a $G$ linearization. Thus the resulting torus $T_{\NS(W)}$ admits a torsor 
$\cQ \ra W$, equivariant under the $G$ action.  

If $X$ admits a universal torsor $\cP \ra X$ then the product
$$\pi_W^* \,\cQ \times \pi_X^*\cP \rightarrow W\times X$$
is a universal torsor for $X\times W$. 

Conversely, suppose that $W\times X$ admits a universal torsor.
Since the existence of a universal torsor is a $G$-birational property, we may assume
that $W=\bP^n$ and $G$ acts linearly and faithfully on $\bP^n$. 
It therefore acts on the associated affine space $\Gamma(\cO_{\bP^n}(1))^{\vee}$ and the universal subbundle
$\cO_{\bP^n}(-1)$.  The No-Name Lemma implies $G$-birational equivalences
$$
\cO_{\bP^n}(-1) \times X \stackrel{\sim}{\dashrightarrow} \bA^1 \times W \times X$$
and 
$$\Gamma(\cO_{\bP^n}(1))^{\vee} \times X \stackrel{\sim}{\dashrightarrow} \bA^{n+1} \times X$$
with trivial actions on the affine space factors. Moreover, $\cO_{\bP^n}(-1)$ is equal to the
blowup of $\Gamma(\cO_{\bP^n}(1))^{\vee}$ at the origin, thus $W\times X$ is stably birational to $\bA^{n+1}\times X$.

We therefore reduce ourselves to the situation where 
$\bP^{n+1}\times X$ admits a universal torsor
$$\cV \rightarrow \bP^{n+1} \times X$$
where $G$ acts trivially on the first factor.   
The pullback
homomorphism
$$\pi^*_X: \Pic(X) \ra \Pic(X\times \bP^{n+1})$$
allows us to produce a $T_{\NS(X)}$-torsor $\cR \rightarrow \bP^{n+1}\times X$. 
Choose a section of $\bP^{n+1}\times X \ra X$ and restrict $\cR$ to this section
to get the desired torsor on $X$.  
\end{proof}

\subsection{Torsors and stable linearization}

We record an equivariant version of 
Proposition~\ref{prop:stabratcrit}. 

\begin{prop}
\label{prop:eq-version}
Let $X$ be a smooth projective $G$-variety with $\Pic(X)=\NS(X)$.
Assume that $X$ admits a $G$-equivariant universal torsor $\cP$ such that 
\begin{itemize}
\item the $G$-action on $\cP$ is stably linearizable,
\item the $G$-action on $T_{\NS(X)}$ is stably linearizable,
\item $\cP\to X$ admits a $G$-equivariant rational section. \end{itemize}
Then the $G$-action on $X$ is stably linearizable. 
\end{prop}
There is no harm in assuming merely that $\cP$ is stably linearizable
as our conclusion on $X$ is a stable property.
\begin{coro}
\label{coro:linearbytorsor}
Let $X$ be a smooth projective $G$-variety with $\Pic(X)=\NS(X)$;
assume $\NS(X)$ is stably a permutation module. If $X$ admits a $G$-equivariant
universal torsor $\cP$ with stably linear $G$-action then the $G$-action on 
$X$ is stably linearizable as well.
\end{coro}
Indeed, the last two conditions of Proposition~\ref{prop:eq-version}
follow if $\NS(X)$ is a stably permutation module by Proposition~\ref{prop:split}. 

\subsection{Universal torsors and Cox rings}
\label{subsect:cox} 
Suppose $X$ is a smooth projective variety that has a universal torsor $\cP \rightarrow X$. 
In some cases, there is a natural embedding of $\cP$ into affine space, realizing
$X$ is a subvariety of a toric variety. 
Specifically, assume that the {\em Cox ring} 
$$
\Cox(X):= \oplus_{L \in \Pic(X)}\, \Gamma(X,L),
$$
graded by the Picard group and with multiplication induced by tensor product of line bundles, is 
finitely generated (see, e.g., \cite{ADH} for definitions and properties). This is the case for Fano varieties, for example \cite{KeelHu,BCHM}.  
Then there is a natural open embedding
$$\cP \hookrightarrow \Spec(\Cox(X)),$$
compatible with the actions of $T_{\NS(X)}$ associated with the torsor structure 
and the grading respectively.
Fixing a finite set $\{x_{\sigma}\}_{\sigma \in \Sigma}$ of graded generators for $\Cox(X)$, we obtain
an embedding 
$$
\Spec(\Cox(X)) \hookrightarrow \bA^{\Sigma}.
$$
Taking a quotient of the codomain by $T_{\NS(X)}$ gives a toric variety (see Section~\ref{sect:toric}); choosing a quotient
associated with a linearization of an ample line bundle $L$ on $X$ gives the desired embedding
$$X \hookrightarrow [\bA^{\Sigma}/T_{\NS(X)}]_L.$$

Our focus is the extent to which these constructions can be performed equivariantly (when $X$
comes with a $G$-action) or over non-closed fields. We emphasize that the Cox-ring formulation
is equivalent to the universal torsor framework when the torsor exists.

\subsection{General results on linearizable actions}
For this last section, we return to the general question of characterizing group actions that
are birational or stably birational.

\begin{prop}
\label{prop:twist}
Let $X$ be a smooth projective variety and $G$ a finite group acting regularly and generically freely on $X$.
Given an automorphism $a:G \rightarrow G$, let ${ }^aX$ denote the resulting twisted action
of $G$ on $X$. If the $G$-action on $X$ is stably linearizable then ${ }^aX$ is stably equivariantly birational to $X$,
hence stably linearizable as well. 
\end{prop}

\begin{proof}
Our assumption implies the existence of linear representations
$$G \times \bA^n \rightarrow \bA^n, \quad G \times \bA^{d+n} \rightarrow \bA^{d+n}, d=\dim(X),$$
such that 
$$
X\times \bA^n\sim_G \bA^{d+n}.
$$
Twisting by $a$, we find that
$${ }^aX \times { }^a \bA^n \sim_G { }^a\bA^{d+n}.
$$
It follows that
$$
X \times { }^a\bA^{d+n} \sim_G { }^aX \times \bA^{d+n}.
$$
The No-Name Lemma implies that these are birational to 
$$X \times \bA^{d+n}, { }^a X\times \bA^{d+n},$$
where the actions on the affine spaces are trivial. This gives the stable birational equivalence.
\end{proof}

\section{Stable linearization of actions on toric varieties}
\label{sect:torvar}

\subsection{Toric varieties}
\label{sect:toric}

Let $X=X_{\Sigma}$ be a $T$-equivariant compactification of $T$, where $\Sigma$ is a {\em fan}, i.e., a collection $\Sigma=\{ \sigma \}$ of cones in the cocharacter group $N:=\mathfrak X_*(T)$ of $T$ (see, e.g., \cite{fulton} for terminology regarding toric varieties). Let $\Sigma(i)\subset \Sigma$ be the collection of $i$-dimensional cones. 
A complete determination of the automorphism group $\Aut(X)$ can be found in \cite{sancho}. Conversely, given a finite group $G\subset \Aut(T)$ there exists a smooth projective $T$-equivariant compactification of $T$, with regular $G$ action. 

Suppose $T$ is a $G$-torus.
We say that $X$ is a {\em $T$-toric variety} if there exists a $G$-equivariant action $X\times T \rightarrow X$
such that $X$ has a dense $T$-orbit with trivial generic stabilizer. Note that $X$ need not have $G$-fixed points but does admit a distinguished Zariski-open subset that is a torsor for $T$.   
  
We record a corollary of Proposition~\ref{prop:twist}:

\begin{coro}
Let $X$ denote a $T$-toric variety that is stably linearizable. 
Given an element $a\in \Aut(X)^G$, the twist ${ }^aX$ is stably linearizable as well and 
$G$-birational of $X$.  
\end{coro}

If the cocharacter module $N$ of $T$ is stably permutation then a smooth projective 
$T$-equivariant compactification $T \subset X$ has Picard group $\Pic(X)$ that
is also stably a permutation module. 

Indeed, we have an exact sequence
\begin{equation} 
\label{eq:SPM} 
0 \rightarrow M \rightarrow \Pic_T(X) \rightarrow \Pic(X), \rightarrow 0,
\end{equation}
where the central term is a permutation module indexed by vectors generating the one-skeleton 
of the fan. 
The exact sequence (\ref{eq:SPM}) shows that $M$ is stably permutation if and only if $\Pic(X)$ is stably permutation.

\subsection{Universal torsors for toric varieties}
\label{sect:cox-toric}

Let $X\times T \rightarrow X$ denote a $T$-toric variety, where $X$ is smooth and projective. 
Ignoring the action of $G$, $\Cox(X)$ is a polynomial ring $k[x_{\sigma}], \sigma \in \Sigma(1),$
indexed by the $1$-skeleton, i.e., generators of the one-dimensional cones in the fan of $X$. 
Of course, the group $G$ permutes the elements of $\Sigma(1)$ and if $X$ admits a $T$-fixed
point -- invariant under $G$ -- then $\Spec(\Cox(X))$ is the affine space $\bA^{\Sigma(1)}$ 
with the induced permutation action of $G$.  

However, when the dense open orbit of $X$ is a nontrivial principal homogeneous space
$$U \times T \rightarrow U$$
it may not be possible to lift the $G$-action compatibly to $\Spec(\Cox(X))$.  
We can identify the cohomology class governing the existence a lifting. Dualizing (\ref{eq:SPM}) gives
$$ 1 \rightarrow T_{\NS(X)} \rightarrow \bG_m^{\Sigma(1)} \rightarrow T \rightarrow 1,$$
encoded by a class $\eta \in \Ext^1_G(T,T_{\NS(X)})$.  
The principal homogeneous space is classified by 
$$[U] \in \rH^1(G,T)$$
and its image under the connecting homomorphism
$$ \partial([U]) = \pm [U] \smile \eta \in \rH^2(G,T_{\NS(X)})$$
is the obstruction to finding a cocycle in $\rH^1(G,\bG_m^{\Sigma(1)})$ 
lifting $[U]$.


\subsection{Actions on $\bP^1$}
\label{sect:p1}

The presence of translations marks an essential discrepancy 
in the analogy between the rationality problem over nonclosed fields and the linearizability problem of actions of finite groups over closed fields, as can be seen from the following example:

Let 
$$
G=\left<\iota_1,\iota_2\right>=\mathfrak C_2 \times \mathfrak C_2
$$ 
and $T$ a one-dimensional torus with $G$ action
$$
\iota_1\cdot t = t^{-1}, \quad \iota_2\cdot t = -t.
$$
Consider an action 
$$\begin{array}{rcl}
T\times \bP^1 & \rightarrow & \bP^1 \\
t\cdot [x,y] & \mapsto & [tx,y].
\end{array}
$$
Let $G$ act on $\bP^1$ by
$$\iota_1 \cdot [x,y]=[y,x], \quad \iota_2 \cdot [x,y]=[-x,y],$$
which is well-defined.  
However, this action does not lift to a linear action of $G$ on $\bA^2$ because
$$\left(\begin{matrix} 0 & 1 \\ 1 & 0 \end{matrix} \right) \left(\begin{matrix} 1 & 0 \\ 0 & -1 \end{matrix} \right) = - \left(\begin{matrix} 1 & 0 \\ 0 & -1 \end{matrix} \right)
\left(\begin{matrix} 0 & 1 \\ 1 & 0 \end{matrix}\right).$$
The Amitsur invariant is
$$
\mathrm{Am}(\bP^1,G)=\bZ/2,
$$
so that this action is not stably linearizable. Alternatively, one may observe that $G$ has no fixed points on $\bP^1$, which is also an obstruction to stable linearizability. 

On the other hand, let 
$$G:=\left<\iota,\sigma:\iota^2=\sigma^3=1, \iota\sigma\iota=\sigma^{-1} \right>
\simeq \fS_3.$$
We continue to have $\iota$ act as $\iota_1$ did above. Let 
$$
\sigma\cdot [x,y]=[\omega x, y], \quad \omega = e^{2\pi i/3}.
$$
This {\em does} lift to a linear action of $G$ on $\bA^2$, e.g., by expressing
$$
\sigma\cdot [x,y]=[\zeta x, \zeta^{-1} y], \quad \zeta = e^{2\pi i/6}.
$$
Again, $G$ has no fixed points on $\bP^1$, but this is {\em not} an obstruction to linearizability, for nonabelian groups.

\subsection{Linearizing actions with translations}

\begin{prop}
\label{prop:linear-torus}
Let $T$ be a $G$-equivariant torus and $X\times T \rightarrow X$ a smooth projective $T$-toric variety.
Assume that
\begin{itemize}
\item{$M=\hat{T}$ is a stably permutation $G$-module;}\
\item{the obstruction $\alpha=\partial(\operatorname{Id}) \in \rH^2(G,T_{\NS(X)})$ vanishes.}
\end{itemize}
Then the $G$-action on $X$ is stably linearizable.
\end{prop}
\begin{proof}
The vanishing assumption shows that $X$ admits a universal torsor $\cP \ra X$ with $G$-action. 
Moreover, we have an open embedding
$$\cP \hookrightarrow \bA^n$$
where $\bA^n$ is an affine space with permutation structure given by the action of $G$ on the
$1$-skeleton of $X$.  

By Proposition~\ref{prop:split} we have $\cP \sim_G T_{\NS(X)}\times X$; the first
factor is stably linearizable by Proposition~\ref{prop:toric:mor}. Since $\cP$ is 
linearizable we conclude $X$ is stably linearizable.
\end{proof}

\begin{ques}
Let $G$ be a finite group, $T$ a $G$-torus, and $X$ a $T$-toric variety. Consider the following conditions: 
\begin{itemize}
\item{the obstruction $\partial(\operatorname{Id}) \in \rH^2(T_{\NS(X)})$ to the existence of a universal torsor vanishes;}
\item{for each $T$-orbit closure $Y\subseteq X$ and subgroup $H\subseteq G$ leaving 
$Y$ invariant, the Amitsur invariant $\mathrm{Am}(Y,H/K)$ vanishes, where $K$ is the subgroup
acting trivially on $Y$.}
\end{itemize}
Are they equivalent?
\end{ques}

Clearly the first implies the second. Recall that the restriction
$$
\Pic(X) \rightarrow \Pic(Y)
$$
can be made to be surjective on a suitable $G$-equivariant smooth projective model of $X$, with induced $T$-closure $Y\subset X$. See, e.g., Sections 2.3--2.5 of \cite{KT-toric}.  

\section{Sextic del Pezzo surfaces}
\label{sect:dp6}

Here we consider actions on the toric surface 
$$
X\subset  \bP^1\times \bP^1\times \bP^1,
$$
given by 
\begin{equation}
\label{eqn:dp6}
X_1X_2X_3 = W_1W_2W_3.
\end{equation}
It has distinguished loci
$$
L_1=\{X_3=W_2=0\}, \, L_2=\{X_1=W_3=0\}, \,
L_3=\{X_2=W_1=0\},
$$
$$
E_{12}=\{X_1=W_2=0\},E_{13}=\{X_3=W_1=0\}, 
E_{23}=\{X_2=W_3=0\}.
$$
Recall that the universal torsor may be realized as an open subset of $\bA^6$ with variables
$$\lambda_1,\lambda_2,\lambda_3,\eta_{12},\eta_{13},\eta_{23},
$$
where
\begin{align*}
X_1=\lambda_2 \eta_{12}, &\quad W_1=\lambda_3 \eta_{13}, \\
X_2=\lambda_3 \eta_{23}, &\quad W_2=\lambda_1 \eta_{12}, \\
X_3=\lambda_1 \eta_{13}, &\quad W_3=\lambda_2\eta_{23}.
\end{align*}
Write 
$$
\Pic(X)=\bZ H + \bZ  E_1 + \bZ E_2 + \bZ E_3
$$
with associated torus 
$$\operatorname{Spec} k[s^{\pm 1}_0,s^{\pm 1}_1,s^{\pm 1}_2,s^{\pm 1}_3]$$
acting via 
$$\lambda_i \mapsto s_i \lambda_i, \quad 
\eta_{ij} \mapsto s_0s_i^{-1}s_j^{-1}\eta_{ij}.$$


\subsection{Action by toric automorphisms}
\label{subsect:autdP6}

Consider the automorphisms of $X$ fixing the distinguished point
$$(1,1,1)=\{X_1=X_2=X_3=W_1=W_2=W_3=1\}.$$
Equivalently, these are induced from automorphisms of the torus
$$T=X \setminus (L_1\cup E_{12}\cup L_2 \cup E_{23} \cup L_3 \cup E_{13}).$$
These are isomorphic to $\fS_2\times \fS_3$ -- we can exchange the $X$ and $W$ variables or permute
the indices $\{1,2,3\}$. 
The induced action on the six-cycle of $(-1)$-curves
may be interpreted as the dihedral group of order $12$.  

Note that the associated exact sequence of $\fS_2\times \fS_3$-modules
$$
0 \rightarrow M \rightarrow \bZ\{(-1)\text{-curves}\} \rightarrow \Pic(X) 
\rightarrow 0
$$
splits.  

\begin{rema}
If $M$ and $P$ are stably permutation $G$-modules then
$\Ext^1_G(P,M)=0$.
This is Lemma 1 in \cite{Sansuc-CT}, which says that if $M$ is coflabby and 
$P$ is permutation then $\Ext^1_G(P,M)=0.$ However, stably permutation modules
are flabby and coflabby \cite[p.~179]{Sansuc-CT}.  
\end{rema}

The $\fS_2\times \fS_3$ action lifts to the Cox ring: For example, let $\fS_3$
act via permutation on the indices and $\fS_2$ by
$$\lambda_i \mapsto \eta_{jk}, \quad  \eta_{jk} \mapsto \lambda_i, \quad \{i,j,k\}=\{1,2,3\}.$$

\subsection{Sextic del Pezzo surface with an $\fS_4$-action}

Assume that $G$ contains nontrivial translations of the torus 
$T=\mathbb G_m^2\subset X$. 
In \cite{sarikyan} it is shown that, on {\em minimal} 
sextic Del Pezzo surfaces, such $G$-actions are not linearizable.

As an example, consider $G:=\fS_4$ acting on $X$ via $\fS_3$-permutations of the factors 
$$
x_1:=X_1/W_1, \quad  x_2:=X_2/W_2, \quad x_3:=X_3/W_3,
$$
and additional involutions (translations)
$$
\iota_1:(x_1,x_2,x_3)\mapsto (-x_1,x_2,-x_3), \quad \iota_2:(x_1,x_2,x_3)\mapsto (-x_1,-x_2,x_3).
$$
Here we have 
$G\cap T(k) = \mathfrak C_2\times \mathfrak C_2$, with $G$ acting on 
$\Aut(N)$ via $\fS_3$. The six exceptional curves form a single $G$-orbit, each curve has generic stabilizer $\mathfrak C_2$ and a nontrivial $\mathfrak C_2$-action.

Using the theory of {\em versal} $G$-covers, 
Bannai-Tokunaga showed that the 
$G$-actions on $\bP^2=\bP(V)$, where $V$ is the standard 3-dimensional representation of $\fS_4$, and on 
\eqref{eqn:dp6}, as described above, are not birational
\cite{bannai}.
Alternative proofs, using the equivariant Minimal Model Program for surfaces, respectively, the Burnside group formalism, can be found in
\cite[Section 3.4]{sarikyan}, respectively \cite[Section 9]{KT-struct}. These approaches cannot be used to study stable linearizability. 

\begin{prop}
\label{prop:dp6-s4}
The $\fS_4$-action is stably linearizable.
\end{prop}
\begin{proof}
We will apply Proposition~\ref{prop:eq-version}, the equivariant version of Proposition~\ref{prop:stabratcrit}.

We use the split sequence 
$$1 \rightarrow \mathfrak C_2 \times \mathfrak C_2 \rightarrow \fS_4 \rightarrow \fS_3 \rightarrow 1$$
induced by (\ref{eqn:aut-tor}) on the $2$-torsion of $T$. 

First, note the action of $G$ on $T_{\NS(X)}$ -- which factors through the homomorphism
$\fS_4 \rightarrow \fS_3$ -- is stably linearizable. 

It suffices then to lift the $G$-action to the Cox ring. The action of $\fS_3$ is clear by the
indexing of our variables.  For the involutions $\iota_1$ and $\iota_2$, we take
$$
\iota_1(\lambda_2)=-\lambda_2
$$
and
$$
\iota_2(\lambda_3)=-\lambda_3,
$$
with trivial action on the remaining variables. The gives the desired lifting.
\end{proof}

\

There is also an action of $G=\fS_3\times \fS_2$ on $X$, with $G\cap T(k)=1$, that is not linearizable, but is stably linearizable. We discuss it in Section~\ref{sect:wg2}.

\section{Weyl group of $\mathsf G_2$ actions}
\label{sect:wg2}
We start with an example presented in \cite[\S~9]{lemire} and motivated by the following question: is the Weyl group action on a maximal torus in a Lie group
equivariantly birational to the induced action on
the Lie algebra of the torus? The authors study 
the action of 
$$
G:=W(\mathsf G_2) \simeq \fS_3\times \fS_2,
$$
the Weyl group of the exceptional Lie group $\mathsf G_2$: 
Consider the torus
$$T= \{(x_1,x_2,x_3): x_1x_2x_3=1 \}$$
and its Lie algebra
$$\ft = \{(y_1,y_2,y_3): y_1+y_2+y_3=0 \},$$
with $\fS_3$ acting on both varieties by permuting the coordinates, 
and $\fS_2 := \left< \epsilon \right>$ acting via
$$
\epsilon \cdot (x_1,x_2,x_3) = (x_1^{-1},x_2^{-1},x_3^{-1})
$$
and 
$$
\epsilon \cdot (y_1,y_2,y_3) = (-y_1,-y_2,-y_3).
$$
We now describe good projective models of both varieties, i.e., such that
the complement of the free locus is normal crossings so that all
stabilizers are abelian.

\subsection{Multiplicative action}
This case builds on section~\ref{subsect:autdP6}; we retain the notation introduced there.

While the sextic del Pezzo surface is a fine model for our group action, it is
often most natural to blow up to eliminate points with nonabelian stabilizers
cf.~\cite[\S 2]{BnG}.
Let $S_{(1,1,1)}$ denote the blowup at $(1,1,1)$.
We identify distinguished loci in $S_{(1,1,1)}$ as proper transforms
of loci in the sextic del Pezzo surface. In addition to the six curves
listed above, we have
\begin{itemize}
\item{$D_i$ from $\{(X_i-W_i)(-1)^{i+1}=0 \}$ for $i=1,2,3$;}
\item{$E$ exceptional divisor over $(1,1,1)$.}
\end{itemize}
The nonzero intersections are
$$E_{12}L_1=E_{12}L_2=E_{23}L_2=E_{23}L_3=E_{13}L_3=E_{13}L_1=1$$
and 
$$D_1L_1=D_1E_{23}=D_1E=1, 
$$
$$
D_2L_2=D_2E_{13}=D_2E=1,
$$
$$
D_3L_3=D_3E_{12}=D_3E=1.$$
All self-intersections are $-1$.  

To compute the Cox ring, we introduce new variables $\delta_i$ and $\eta$ associated with
$D_i$ and $E$. The resulting relations are
\begin{align*}
\delta_1\eta &=X_1-W_1 = \lambda_2\eta_{12} - \lambda_3 \eta_{13}, \\
\delta_2 \eta &=-X_2+W_2=-\lambda_3\eta_{23} + \lambda_1 \eta_{12}, \\
\delta_3 \eta &=X_3-W_3=\lambda_1 \eta_{13} - \lambda_2\eta_{23}.
\end{align*}
Reassigning
$$\lambda_i=p_{i4}, \eta_{ij}=p_{k5}, \delta_i=p_{jk}, \eta=p_{45}$$
we obtain three Pl\"ucker relations.  The remaining relations
$$p_{12}p_{34}-p_{13}p_{24}+p_{14}p_{23}=p_{12}p_{35}-p_{13}p_{25}+p_{15}p_{23}=0$$
are also valid.  

The group $\fS_3\times \fS_2$ may be interpreted as permutations of the
sets $\{1,2,3\}$ and $\{4,5\}$. In the natural induced action, 
$$(ij)\cdot p_{ij} = -p_{ij}, \quad \epsilon\cdot p_{45} = -p_{45}$$
but the actions on the original six variables are compatible.

The elements
$$(\zeta,\zeta,\zeta), \,\,(\zeta^2,\zeta^2,\zeta^2)  \in T, \quad
\zeta = e^{2 \pi i/3},
$$ 
are fixed by $\fS_3$.  
The curves in the sextic del Pezzo surface
$$
F_{12}=\{X_1W_2-W_1X_2=0\}, 
$$
$$
F_{13}=\{X_1W_3 - W_1 X_3=0 \},
$$
$$
F_{23}=\{X_2W_3-W_2X_3=0\}$$
meet at the three diagonal points and have intersections
$$F_{12}^2=F_{13}^2=F_{23}^2=2, \quad
 F_{12}F_{13}=F_{12}F_{23}=F_{13}F_{23}=3.$$
Let $S_{\times} \rightarrow S_{(1,1,1)}$ denote the blowup at these points,
a cubic surface.

Iskovskikh \cite{isk-s3} presents an equivariant birational morphism
$$S_{(1,1,1)} \rightarrow Q=\{3\hat{w}^2=xy+xz+yz\} \subset \bP^3$$
obtained by double projection of the sextic del Pezzo from 
$(1,1,1)$. This blows down the proper transforms of $D_1,D_2,$ and $D_3$.
Here $\fS_3$ acts by permutation of $\{x,y,z\}$ and $\epsilon \cdot w= -w$.
Indeed, the proper transforms of $L_1,L_2,L_3$ are in one ruling;
the proper transforms of $E_{23},E_{13},E_{12}$ are in the other ruling.

This can be obtained as follows: Choose a basis for the forms
vanishing to order two at $(1,1,1)$:
\begin{align*}
x &= (X_1+W_1)(X_2-W_2)(X_3-W_3) \\
y &= (X_1-W_1)(X_2+W_2)(X_3-W_3) \\
z &= (X_1-W_1)(X_2-W_2)(X_3+W_3) \\
w &= (X_1-W_1)(X_2-W_2)(X_3-W_3) 
\end{align*}
so we have
$$xy+xz+yz=w(2(X_1X_2X_3-W_1W_2W_3)+w)\equiv w^2.$$
We use (\ref{eqn:dp6}) to get the last equivalence on our degree-six
del Pezzo surface.

\subsection{Additive action}
We turn to the action on the Lie algebra:
The representation of $\ft$ is linear and admits a compactification
$$
\ft \subset \bP(\ft \oplus k).
$$
Write $y_1=Y_1/Z$ and $y_2=Y_2/Z$ 
so that the induced action on $\bP^2$ has fixed point $[0,0,1]$ 
and distinguished loci
$$A_{12}=\{Y_1=Y_2\}, \quad A_{13}= \{Y_1=-Y_1-Y_2 \}, \quad A_{23} = \{Y_2=-Y_1-Y_2\}$$
and 
$$B_{12}=\{Y_2=-Y_1\}, \quad B_{13}= \{Y_2=0\}, \quad B_{23}=\{Y_1=0\}.$$
Blowing up the origin $Y_1=Y_2=0$ yields a smooth projective surface
$\simeq \mathbb{F}_1$ with abelian stabilizers.

The Cox ring is given by
$$
k[\zeta, \beta_{13},\beta_{23}, \eta],
$$
with $Z=\zeta$, $Y_1=\eta \beta_{23}$, $Y_2=\eta \beta_{13}$. 
One lift of the $\fS_3\times \fS_2$-action has $\fS_3$ acting with the
standard two-dimension representation on $\beta_{13},\beta_{23}$
and $\fS_2$-action via $\epsilon\cdot \eta=-\eta$.  
The two-dimensional torus acts via 
$$(\eta,\beta_{13},\beta_{23},\zeta) \mapsto 
(t_E \eta, t_f \beta_{13}, t_f \beta_{23}, t_E t_f \zeta).$$

\subsection{On the Lemire-Reichstein-Popov stable equivalence  \cite{lemire}}
\label{sect:lem-reich}
Consider the rational map
$$
\begin{array}{rcl}
\ft &\dashrightarrow& \bP(\ft) \\
(Y,Z) & \mapsto & [Y,Z]. 
\end{array}
$$
Taking Cartesian products, we obtain
$$
\begin{array}{rcl}
\ft \times \ft \simeq \bA^4  &\dashrightarrow& \bP(\ft) \times \bP(\ft) \\
(Y_1,Z_1,Y_2,Z_2) & \mapsto & ([Y_1,Z_1],[Y_2,Z_2]). 
\end{array}
$$
This induces a rank-two vector bundle
$$\Bl_{\{Y_1=Z_1=0\}\cup \{Y_2=Z_2=0\}}(\bA^4) \rightarrow {\bP(\ft)}^2.$$

We take the product as an $\fS_3\times \fS_2$-variety, where the first factor
acts diagonally and the second factor interchanges the two factors.  
Thus ${\bP(\ft)}^2\simeq Q$ as $\fS_3\times \fS_2$-varieties. 

On the other hand, there is a morphism 
$$
\begin{array}{rcl}
\bA^4 & \rightarrow & \ft \\
(Y_1,Z_1,Y_2,Z_2) & \mapsto & (Y_1-Y_2,Z_1-Z_2)
\end{array}
$$
which is also a rank-two vector bundle over $\ft$.  

Applying the No-Name Lemma twice, we conclude that 
$\ft \times \bA^2$ and $T\times \bA^2$ -- with trivial actions on the $\bA^2$ factors
-- are $G$-equivariantly birational to each other. 

\

\noindent
{\bf Question:} Is the affine quadric threefold
$$ w^2 = xy+xy+yz $$
$G$-equivariantly birational equivalent to $\ft \times \bA^1$?

\section{Quadric surfaces}
\label{sect:quad}

We are now in a position to settle the stable linearizability problem for quadric surfaces 
$$
X=\bP^1\times \bP^1,
$$
completing the results in \cite[Thm. 3.25]{sarikyan}, which identifies linearizable actions. 

Let $G$ act generically freely and minimally on $\bP^1 \times \bP^1$. In particular, there
exist elements exchanging the two factors. Let $G_0$ be the intersection of $G$ with the
identity component of 
$$\Aut(\bP^1)^2 \subset \Aut(\bP^1 \times \bP^1),$$ 
so we have an exact sequence
$$1 \rightarrow G_0 \rightarrow G \rightarrow \fS_2 \rightarrow 1.$$
Each element 
$\iota \in G \setminus G_0$ acts via conjugation $G_0$.  
Let $D$ denote the intersection of $G_0$ with the diagonal subgroup and
$A_i$ the image of $G_0$ under the projection $\pi_i$. Conjugation by $\iota$
takes the kernel of $G_0\rightarrow A_1$ to the kernel of $G_0 \rightarrow A_2$
and thus induces an isomorphism
$$\phi_{\iota}:A_1 \stackrel{\sim}{\rightarrow} A_2$$
restricting to the identity on $D$.  

Sarikyan shows that $G$ is linearizable if and only if $A\simeq \mathfrak C_n$, the cyclic group \cite[Lemma 3.24]{sarikyan}.  
Moreover, 
\begin{itemize}
\item{the only linearizable actions of $A$ on $\bP^1$ are by 
$\mathfrak C_n$ or $\mathfrak D_n$, the dihedral group of order $2n$, with $n>1$ odd;}
\item{the remaining group actions on $\bP^1$ cannot be linearized due to the
Amitsur obstruction.}
\end{itemize}
Thus the only possible candidate for {\em stably} linearizable but nonlinearizable
actions on $\bP^1\times \bP^1$ are when $A\simeq \mathfrak D_n$, $n>1$ odd.  

\begin{prop}
\label{prop:quads}
Under the assumptions above, $G$-actions on $\bP^1\times \bP^1$ with $A\simeq \mathfrak D_n$, with $n>1$ odd, are
always stably linearizable. 
\end{prop}

\begin{proof}
Suppose that $\bP^1 \times \bP^1 = \bP(V_1) \times \bP(V_2)$, where $V_1$ and $V_2$
are representations of $A_1$ and $A_2$, along with an isomorphism of $D$-representations
$$
V_1|D \stackrel{\sim}{\ra} V_2|D.
$$
Using the quotient $G_0 \twoheadrightarrow A_1$, we can regard $V_1$ as a representation
of $G_0$.  Take the induced representation
$$\operatorname{Ind}_{G_0}^G (V_1)$$
which has dimension four.  Mackey's induced character formula implies
that the restriction of this representation back down to $G_0$ is
of the form
$$V_1 \oplus V_2,$$
where $V_2$ is regarded as a $G_0$ representation via $G_0\twoheadrightarrow A_2$.

Now $V_1 \oplus V_2$, as a variety, is the product $V_1\times V_2$.  
The rational maps $V_i \dashrightarrow \bP(V_i)$ induce 
$$V_1 \times V_2 \dashrightarrow \bP(V_1) \times \bP(V_2),$$
resolved by blowing up $\{0\}\times V_2$ and $V_1 \times \{0\}.$
This has the structure of a rank-two $G$-equivariant vector bundle. 
The No-Name Lemma implies that $V_1\times V_2$ is birational to
$\bA^2 \times \bP(V_1)\times\bP(V_2)$ where the first factor has
trivial $G$-action. Hence the $G$-action on $\bP(V_1)\times \bP(V_2)$ is 
stably linearizable.
\end{proof}

For $G=W(\mathsf G_2)=\fS_2\times \fS_3$ this is precisely the result
of \cite[\S 9]{lemire} presented in Section~\ref{sect:lem-reich}.

\subsection{Generalizations}

The same argument gives:
\begin{prop}
Let $G$ be a finite group acting generically freely
on $(\bP^m)^r$. Write $G_0\subset G$ for
the intersection of $G$ with the identity component of $\Aut((\bP^m)^r)$.
Suppose that 
\begin{itemize}
\item{$G$ acts transitively on the $r$ factors;}
\item{the image $A_i$ of $\pi_i: G_0 \ra \Aut(\bP^m)$, the projection to the $i$-th factor, has a linearizable
action on $\bP^m$.}
\end{itemize}
Then the action of $G$ on $(\bP^m)^r$ is stably linearizable.
\end{prop}

\begin{prop}
Let $G$ be a finite group. Let $G$ act generically freely on smooth projective varieties $X_1$ and $X_2$ with $\Pic(X_i)=\NS(X_i)$. Suppose there
exist universal torsors $\cP_i \ra X_i$ with compatible $G$ actions.  
Then 
$$\cU:=\pi_1^*\,\cP_1 \times_{X_1\times X_2} \pi_2^*\,\cP_2 \rightarrow X_1\times X_2$$
is a universal torsor as well. 

If $\NS(X_1)\oplus \NS(X_2)$ is a stably permutation module then
$X_1\times X_2$ is stably birational to $\cU$.  

Moreover, if the $X_i$ are $T_i$-toric varieties then 
$X_1 \times X_2$ is stably linearizable.  
\end{prop}

\section{Quotients of flag varieties by tori}
\label{sect:quotor}

\subsection{Weyl group actions on Grassmannians}
Consider the Grassmannian $\Gr(m,n)$ of $m$-dimensional subspaces of an $n$-dimensional vector space.  
Once we fix a basis for the underlying vector space, the symmetric group $\mathfrak{S}_n$ acts naturally on
$\Gr(m,n)$.  

Every element of $\Gr(m,n)$ may be interpreted as the span of the rows of an $m\times n$ matrix $A$
with full rank.  Let $\bA^{mn}$ denote the affine space parametrizing these and $U \subset \bA^{mn}$
the open subset satisfying the rank condition. Then
$$
\Gr(m,n) = \GL_m \backslash U,
$$
where the linear group acts via multiplication from the left.  
Let 
$$
\mathcal S \rightarrow \Gr(m,n)
$$ 
denote the universal subbundle of rank $m$,
$\End(\mathcal S)=\mathcal S^*\otimes \mathcal S$, and $\GL(\mathcal S) \subset \End(\mathcal S)$ the associated frame/principal $\GL_m$ bundle.
We write the induced $\GL_m$-action on $\GL(\mathcal S)$ from the left.    
Note that 
$$\dim \GL(\mathcal S) = \dim \Gr(m,n) + \rank(\mathcal S)^2 = m(n-m)+m^2;$$
indeed, we may identify $\GL(\mathcal S)$ with $U$, equivariantly with respect to the natural left 
$\GL_m$ actions.  

Returning to the $\mathfrak{S}_n$-action: It acts on the $m\times n$ matrices by permuting the
columns, which commutes with the $\GL_m$-action given above.  In particular the action is linear on $\bA^{mn}$. 
This action coincides with the natural induced action on $S$, $\End(\mathcal S)$, and $\GL(\mathcal S)$.  The No-Name Lemma says that
the $\fS_n$-action on $\End(\mathcal S)$ -- regarded as a vector bundle over $\Gr(m,n)$ -- is equivalent to the action on 
$\bA^{m^2} \times \Gr(m,n)$ with trivial action on the first factor. 
We conclude:

\begin{prop}
\label{prop:st-grass}
The action of $\mathfrak{S}_n$
on $\Gr(m,n)$ is stably linearizable.  
\end{prop}

\subsection{Del Pezzo surface of degree 5}

It is well-known that a del Pezzo surface of degree 5 can be viewed 
as the moduli space $\overline{\cM}_{0,5}$ of 5 points on $\bP^1$ and thus carries a  natural action of $\fA_5$, induced from the action of $\fS_5$ on the points (see, e.g.,  \cite[Section 1]{sarikyan}). It is also known that this $\fA_5$-action is not linearizable (see e.g., \cite{bannai} or \cite[Theorem 6.6.1]{CS}). 
Again, this should be contrasted with the situation over nonclosed fields, where {\em all} degree 5 del Pezzo surfaces are rational. 

Consider a three-dimensional irreducible faithful representation
$$\varrho: \fA_5 \rightarrow \GL(V).$$
There are two such representations, which are dual to each other.
This gives rise to a generically free (linear!) action of $\fA_5$ on
$\bP^2$. The two linear actions on $\bP^2$ are not conjugated in $\PGL_3$, but {\em are} equivariantly birational \cite[Remark 6.3.9]{CS}.

As an application of Proposition~\ref{prop:st-grass}, we obtain:

\begin{prop}
\label{prop:a5}
The $\fA_5$-actions on $\bP^2$ and $\overline{\cM}_{0,5}$ are not birational but stably birational.
\end{prop}
\begin{proof}
It suffices to show that the action of $\fA_5$ on $\overline{\cM}_{0,5}$ is stably linear. We have seen already that the action on 
the Grassmannian $\Gr(2,5)$ is stably linear. We are using that the N\'eron-Severi torus acts on the cone over $\Gr(2,5)$ with quotient $\overline{\cM}_{0,5}$.
Proposition~\ref{prop:split}
gives the desired result once we check that $\NS(\overline{\cM}_{0,5})$
is stably permutation. We may write
$$M:=\NS(\overline{\cM}_{0,5})=\bZ L + \bZ E_1 + \bZ E_2 + \bZ E_3 + \bZ E_4 $$
so that the $\fS_4$-action is clear. 
The transposition $(45)$ may be realized
by the Cremona map acting by:
\begin{align*}
L & \mapsto 2L-E_1-E_2-E_3\\
E_1 & \mapsto L-E_2-E_3 \\
E_2 & \mapsto L-E_1-E_3 \\
E_3 & \mapsto L-E_1-E_2 \\
E_4 &\mapsto E_4 
\end{align*}
Introducing the auxiliary $\bQ$-basis
\begin{align*}
L_5&=L,\\
L_4&=2L-E_1-E_2-E_3,\\
L_3&=2L-E_1-E_2-E_4,\\
L_2&=2L-E_1-E_3-E_4,\\
L_1&=2L-E_2-E_3-E_4,
\end{align*}
we see immediately that this {\em submodule} 
$\left<L_1,L_2,L_3,L_4,L_5\right>$ is a permutation module.

Consider the direct sum $M\oplus (\bZ F_1 \oplus \bZ F_2)$ where the 
action on the second factor is trivial. This decomposes over $\bZ$
into summands
$$
\left<L_1\!-\!F_1\!-\!F_2,L_2\!-\!F_1\!-\!F_2, L_3\!-\!F_1\!-\!F_2, L_4\!-\!F_1\!-\!F_2,L_5\!-\!F_1\!-\!F_2\right>
$$
and
$$\left<3L\!-\!E_1\!-\!E_2\!-\!E_3\!-\!E_4\!-\!F_1\!-\!2F_2, 3L\!-\!E_1\!-\!E_2\!-\!E_3\!-\!E_4\!-\!2F_1\!-\!F_2 \right>.
$$
The first is a permutation module and the second is trivial.   
\end{proof}

\subsection{Segre cubic threefold}
\label{sect:mo6}

There are two nonconjugate embeddings of $\fA_5$ into $\fS_6$, differing by the nontrivial
outer automorphism of $\fS_6$ \cite[\S 1]{HMSV}. Thus we obtain two actions of 
$G:=\fA_5$ on the Segre cubic threefold $X_3$, hence on $\overline{\cM}_{0,6}$. 
It is known that one of the actions (the nonstandard one) is $G$-equivariant to a linear action on $\bP^3$ \cite[Ex. 1.3.4]{CS}, and that the other is {\em birationally superrigid}, in particular, not linearizable \cite[Theorem 4.8]{avilov-cubic}. 

Regarding $\NS(\overline{\cM_{0,6}})$ as a $G$-module for the nonstandard action, we see that it is stably a permutation module -- since this action is linearizable.  
However, for any finite group $G$ and automorphism $a:G\rightarrow G$, precomposing 
by $a$ yields an action on $G$-modules; this respects permutation and stably permutation modules.  
It follows that the ``standard'' action on $\NS(\overline{\cM_{0,6}})$ is also a stably permutation 
module. 

Consider the class group $\operatorname{Cl}(X_3)$ and  $\NS(\overline{\cM}_{0,6})$ as $\fS_6$-modules. These differ by a permutation module, namely, partitions of $\{1,2,3,4,5,6\}$ into unordered pairs of subsets of size three. Recall that  
$X_3$ is a quotient of $\Gr(2,6)$ by the maximal torus $T\subset \GL_6$. 
The torus acting on the cone over $\Gr(2,6)$ is not the N\'eron-Severi torus for $\overline{\cM_{0,6}}$; it is the N\'eron-Severi torus for small resolutions of $X_3$ -- or even for $X_3$ itself if we
allow Weil divisors on $X_3$. 
The {\em standard} action of $\fS_6$, and thus also of $\fA_5$, on $\Gr(2,6)$ is stably linearizable by Proposition~\ref{prop:st-grass}.
We conclude:

\begin{prop}
\label{prop:segre}
The standard and the nonstandard actions of $\fA_5$ on the Segre cubic threefold are not birational but stably birational. 
\end{prop}

\begin{rema}
Florence and Reichstein \cite{FlorenceReichstein} consider, over nonclosed fields, the rationality 
of twists of $\overline{\cM}_{0,n}$ arising from automorphisms associated with permutations of the marked points.
These are always rational for odd $n$ but may be irrational when $n$ is even.  
\end{rema}

\bibliographystyle{alpha}
\bibliography{gtorsor}

\end{document}